\documentclass[12pt,reqno]{amsart}%
\usepackage{amssymb}
\usepackage{amsfonts}
\usepackage{amsmath}
\usepackage{graphicx}
\usepackage{geometry}%
\newtheorem{theorem}{Theorem}
\newtheorem*{thmK}{Theorem K}
\newtheorem*{thmM}{Theorem M}
\theoremstyle{plain}

\newtheorem{case}{Case}

\newtheorem{corollary}{Corollary}

\newtheorem{lemma}{Lemma}

\numberwithin{equation}{section}
\begin{document}
\title[An inequality for Kruskal-Macaulay functions]{An inequality for Kruskal-Macaulay functions}
\author{Bernardo M. \'{A}brego}
\address[B. M. \'{A}brego and S. Fern\'{a}ndez-Merchant]{Department of Mathematics\\
California State University, Northridge at 18111 Nordhoff Street, Northridge,
CA 91330}
\email{bernardo.abrego@csun.edu and silvia.fernandez@csun.edu}
\author{Silvia Fern\'{a}ndez-Merchant}
\curraddr{(B. M. \'{A}brego and S. Fern\'{a}ndez-Merchant) \textsc{Centro de
Investigaci\'{o}n en Matem\'{a}ticas, A.C., Guanajuato, Gto., Mexico}}
\author{Bernardo Llano}
\address[B. Llano]{Departamento de Matem\'{a}ticas \\
Universidad Aut\'{o}noma Metropolitana, Iztapalapa, San Rafael Atlixco 186,
Colonia Vicentina, 09340, M\'{e}xico, D.F.}
\email{llano@xanum.uam.mx}
\thanks{This paper is in final form and no version of it will be submitted for
publication elsewhere.}
\date{January 22, 2009}
\subjclass[2000]{Primary 05A05; Secondary 05A20}
\keywords{binomial representation of a positive integer, Kruskal-Macaulay function,
shadow of a set}

\begin{abstract}
Given integers $k\geq1$ and $n\geq0$, there is a unique way of writing $n$ as
$n=\binom{n_{k}}{k}+\binom{n_{k-1}}{k-1}+...+\binom{n_{1}}{1}$ so that $0\leq
n_{1}<\cdots<n_{k-1}<n_{k}$. Using this representation, the
\emph{Kruskal-Macaulay function of }$n$ is defined as $\partial^{k}\left(
n\right)  =\binom{n_{k}-1}{k-1}+\binom{n_{k-1}-1}{k-2}+...+\binom{n_{1}-1}%
{0}.$ We show that if $a\geq0$ and $a<\partial^{k+1}\left(  n\right)  $, then
$\partial^{k}\left(  a\right)  +\partial^{k+1}\left(  n-a\right)  \geq
\partial^{k+1}\left(  n\right)  .$ As a corollary, we obtain a short proof of
Macaulay's Theorem. Other previously known results are obtained as direct consequences.

\end{abstract}
\maketitle

\section{Introduction}

Given integers $k\geq1$ and $n\geq0$, there is a unique way of writing $n$ as
\begin{equation}
n=\binom{n_{k}}{k}+\binom{n_{k-1}}{k-1}+...+\binom{n_{2}}{2}+\binom{n_{1}}{1}
\label{binomial rep}%
\end{equation}
so that $0\leq n_{1}<n_{2}<\cdots<n_{k-1}<n_{k}$. Using this representation,
called the $k$\emph{-binomial representation of }$n$, the
\emph{Kruskal-Macaulay function of }$n$ is defined as%
\[
\mathbf{\partial}^{k}\left(  n\right)  =\binom{n_{k}-1}{k-1}+\binom{n_{k-1}%
-1}{k-2}+...+\binom{n_{2}-1}{1}+\binom{n_{1}-1}{0}\text{.}%
\]
(see \cite{And, CL69, Knu05, Mac27} for details.) The main goal of this paper
is to prove the following inequality for Kruskal-Macaulay functions and show
some of its consequences.

\begin{theorem}
\label{technical}Let $k,a,$ and $n$ be integers such that $k\geq1$ and $n\geq
a\geq0$. If $a<$ $\mathbf{\partial}^{k+1}\left(  n\right)  $, then%
\begin{equation}
\mathbf{\partial}^{k}\left(  a\right)  +\mathbf{\partial}^{k+1}\left(
n-a\right)  \geq\mathbf{\partial}^{k+1}\left(  n\right)  .
\label{technical inequality}%
\end{equation}
Moreover, if $n=\binom{N}{k+1}$ for some $N\geq k+1,$ then the equality in
(\ref{technical inequality}) occurs only when $a=0.$
\end{theorem}

Kruskal-Macaulay functions are relevant for their applications to the study of
antichains in multisets (see for example \cite{Knu05, And}), posets, rings and
polyhedral combinatorics (see \cite{BB97} and the survey \cite{BL04}). In
particular, they play and important role in proving results, extensions and
generalizations of classical problems concerning the Kruskal-Katona
\cite{Kru63, Kat66, Wegner}, Macaulay \cite{Mac27}, and Erd\H{o}s-Ko-Rado
\cite{EKR61} theorems. More recently, the authors \cite{AFL07}, applied
Theorem \ref{technical} to the problem of finding the maximum number of
translated copies of a pattern that can occur among $n$ points in a
$d$-dimensional space, a typical problem related to the study of repeated
patterns in Combinatorial Geometry. For every $P\subseteq\mathbb{R}^{d}$, a
fixed finite point set (called a \textit{pattern}) we say that $P$ is a
\emph{rational simplex} if all the points of $P$ are rationally affinely
independent. In \cite{AFL07}, we proved that the maximum number of translated
copies of a rational simplex $P$ with $|P|=k+1$, determined by a set of $n$
points of $\mathbb{R}^{d}$, is equal to $n-\mathbf{\partial}^{k}\left(
n\right)  .$

We now introduce some terminology needed to state the Kruskal-Katona and
Macaulay Theorems. Let $M_{k}$ and $S_{k}$ denote the set of nonincreasing,
respectively decreasing, sequences of natural integers of length $k$, i.e.,%
\begin{align*}
M_{k}  &  =\left\{  \left(  x_{1},x_{2},...,x_{k}\right)  \in\mathbb{N}%
^{k}:x_{1}\geq x_{2}\geq...\geq x_{k}\geq1\right\} \\
S_{k}  &  =\left\{  \left(  x_{1},x_{2},...,x_{k}\right)  \in\mathbb{N}%
^{k}:x_{1}>x_{2}>...>x_{k}>1\right\}  \text{.}%
\end{align*}
If $A\subseteq M_{k}$ (or $S_{k})$, then the \emph{shadow} of $A$, denoted by
$\partial A$, consists of all nonincreasing (decreasing) subsequences of
length $k-1$ of elements of $A$ ($\partial(\emptyset)=\emptyset$). That is,%
\[
\partial A=\left\{  x:x\text{ is a subsequence of }y\text{ of length
}k-1\text{, for some }y\in A\right\}  .
\]
By analogy, one can think of $M_{k}$ (or $S_{k}$) as multisets (or sets) of
size $k$, with positive integers as elements. In this context $\partial A$
consists of the subsets of multisets (or sets) in $A$ of cardinality $k-1$.

The Kruskal-Katona function, $\partial_{k}$ (defined below), is the analogue
of the Kruskal-Macaulay function defined before. For $n$ as in
(\ref{binomial rep}),%
\[
\mathbf{\partial}_{k}\left(  n\right)  =\binom{n_{k}}{k-1}+\binom{n_{k-1}%
}{k-2}+...+\binom{n_{2}}{1}+\binom{n_{1}}{0}\text{.}%
\]
The sets of sequences $M_{k}$ and $S_{k}$ are \emph{lexicographically
ordered}. That is, for $x$ and $y$ in $M_{k}$ (or $S_{k}$), $x\prec y$ if for
some index $i$, $x_{i}<y_{i}$ and $x_{j}=y_{j}$ whenever $j<i$. There is an
important relationship between shadows of multisets and sets and the functions
$\partial^{k}$ and $\partial_{k}$. Namely, if we denote by $FM_{k}(n)$ and
$FS_{k}(n)$ the first $n$ members, in lexicographic order, of $M_{k}$ and
$S_{k}$, respectively; then%
\begin{equation}
\left\vert \partial FM_{k}\left(  n\right)  \right\vert =\mathbf{\partial}%
^{k}\left(  n\right)  \text{ and }\left\vert \partial FS_{k}\left(  n\right)
\right\vert =\mathbf{\partial}_{k}\left(  n\right)
.\label{shadow vs macaulay}%
\end{equation}

The Kruskal-Katona and Macaulay Theorems show that in fact $\partial_{k}(n)$
and $\partial^{k}(n)$ are the best lower bounds for the shadow of a set with
$n$ elements,

\begin{thmK}
\emph{(Kruskal \cite{Kru63}-Katona \cite{Kat66})} Let $k\geq0$; for every
$A\subseteq S_{k+1}$,%
\[
\left\vert \partial A\right\vert \geq\left\vert \partial FS_{k+1}\left(
\left\vert A\right\vert \right)  \right\vert =\partial_{k+1}\left(  \left\vert
A\right\vert \right)  \text{.}%
\]

\end{thmK}

\begin{thmM}
\emph{(Macaulay \cite{Mac27})} Let $k\geq0$; for every $A\subseteq M_{k+1}$,%
\[
\left\vert \partial A\right\vert \geq\left\vert \partial FM_{k+1}\left(
\left\vert A\right\vert \right)  \right\vert =\partial^{k+1}\left(  \left\vert
A\right\vert \right)  \text{.}%
\]

\end{thmM}

We present, in Section \ref{consequences}, a short and simple proof of Theorem
M obtained as a corollary of Theorem \ref{technical}. We point out that
Eckhoff and Wegner \cite{EW}, ( see also Daykin \cite{Daykin}) obtained a
proof of Theorem K as a consequence of an inequality similar to
(\ref{technical inequality}). Namely, for $n\geq a\geq0$,%
\begin{equation}
\max\left(  \partial_{k}(a),n-a\right)  +\partial_{k+1}\left(  n-a\right)
\geq\partial_{k+1}\left(  n\right)  \text{.} \label{kruskal ineq}%
\end{equation}
The equivalent inequality for the functions $\partial^{k}$ and $\partial
^{k+1}$ is true, and it was in fact generalized by Bj\"{o}rner and Vre\'{c}ica
\cite{BB97} to a larger number of terms (see Corollary \ref{Bjorner}). The
proof of their result depends on Macaulay's Theorem. However, we are not aware
of, nor could we find, a proof of Theorem M obtained as a consequence of this
result. We show, in Section \ref{consequences}, how Bj\"{o}rner and
Vre\'{c}ica's inequalities follow easily from Theorem \ref{technical}.

Our proof of the theorem, presented in Section 3, is elementary as it only
relies on properties of binomial coefficients. Some of the ideas are similar
to those used in \cite{EW} for the proof of (\ref{kruskal ineq}).

The condition $a<\partial^{k+1}(n)$ in Theorem \ref{technical} cannot be
strenghtened. For instance, whenever $k\geq2,$ $n_{3}=4,$ $n_{2}=2,$
$n_{1}=1,$ and $a=$ $\mathbf{\partial}^{k+1}\left(  n\right)  ,$ we have that
\[
\mathbf{\partial}^{k}\left(  a\right)  +\mathbf{\partial}^{k+1}\left(
n-a\right)  =\mathbf{\partial}^{k+1}\left(  n\right)  -1<\mathbf{\partial
}^{k+1}\left(  n\right)  .
\]

Finally, it is an interesting open problem to determine the pairs $(n,a)$ with
$a<\partial^{k+1}(n)$ that achieve equality in (\ref{technical inequality}).
So far we were able to classify the pairs when $n$ is of the form $\binom
{N}{k+1}$. The solution to this problem would be the first step to classify
all patterns $P$ for which the maximum number of translates of $P$, among $n$
points in $\mathbb{R}^{d}$; is equal to $n-\partial^{k}(n)$.

\section{Consequences of the theorem\label{consequences}}

We first prove Macaulay's Theorem as a corollary of Theorem \ref{technical}.

\begin{proof}
[Proof of Theorem M]Let $A\subseteq M_{k+1}$. We proceed by induction on
$k+\left\vert A\right\vert $. If $k=0$ or $A=\emptyset$, the result is
trivially true. Suppose $k\geq1$ and $A\neq\emptyset$. Set $A_{11}=\{x\in
M_{k}:x_{k}=1$ and $x\ast1\in A\}$, $A_{12}=\{x\in M_{k}:x_{k}\geq2$ and
$x\ast1\in A\}$, and $A_{2}=\{x\in A:x_{k+1}\geq2\}$. Here $x\ast1$ denotes
the concatenation of $x$ and $1$, that is $x\ast1$ is the $k$-tuple $x$ with
an entry 1 appended in the $(k+1)^{\text{th}}$ position. Clearly,
$A=(A_{11}\ast1)\cup(A_{12}\ast1)\cup A_{2}$ and the terms in the union are
pairwise disjoint. Moreover, we can assume that $A_{11}\cup A_{12}%
\neq\emptyset$. Otherwise, since all entries of members of $A$ are $\geq2$, we
can work with the set $A^{\prime}$ obtained by subtracting 1 to every entry in
the sequences of $A$ ($\left\vert A^{\prime}\right\vert =\left\vert
A\right\vert $ and $\left\vert \partial A^{\prime}\right\vert =\left\vert
\partial A\right\vert $.) Let $a=\left\vert A_{11}\right\vert +\left\vert
A_{12}\right\vert $ and $b=\left\vert A_{2}\right\vert $. Note that
$\left\vert A\right\vert =a+b$ and $a\geq1$.

If $x=(x_{1},x_{2},\ldots,x_{k})\in A_{11}$, then $(x_{1},x_{2},\ldots
,x_{k-1})\in\partial A_{11}$ and $(x_{1},x_{2},\ldots,x_{k-1},1)=x\in\partial
A_{11}\ast1$. That is, $A_{11}\subseteq\partial A_{11}\ast1$. We now calculate
$\partial A$ in terms of $A_{11},A_{12}$, and $A_{2}$. We use the property
that $\partial(A\cup B)=\partial A\cup\partial B$.
\begin{align*}
\partial A  &  =\partial A_{2}\cup A_{12}\cup A_{11}\cup\left(  \partial
A_{11}\ast1\right)  \cup\left(  \partial A_{12}\ast1\right) \\
&  =\partial A_{2}\cup A_{12}\cup\left(  \partial A_{11}\ast1\right)
\cup\left(  \partial A_{12}\ast1\right) \\
&  =(\partial A_{2}\cup A_{12})\cup\left(  \partial(A_{11}\cup A_{12}%
)\ast1\right)  .
\end{align*}
If $x\in(\partial A_{2}\cup A_{12})$, then $x_{k}\geq2$. Thus
\[
(\partial A_{2}\cup A_{12})\cap\left(  \partial(A_{11}\cup A_{12}%
)\ast1\right)  =\emptyset,
\]
and consequently%
\begin{equation}
\left\vert \partial A\right\vert =\left\vert \partial A_{2}\cup A_{12}%
\right\vert +\left\vert \partial(A_{11}\cup A_{12})\right\vert \text{.}
\label{macaulay proof identity}%
\end{equation}
We consider two cases. If $a\geq$ $\mathbf{\partial}^{k+1}(\left\vert
A\right\vert )$, then
\[
\left\vert \partial A\right\vert =\left\vert \partial A_{2}\cup A_{12}%
\right\vert +\left\vert \partial(A_{11}\cup A_{12})\right\vert \geq\left\vert
A_{12}\right\vert +\left\vert A_{11}\right\vert =a\geq\mathbf{\partial}%
^{k+1}(\left\vert A\right\vert ).
\]
Assume $a<$ $\mathbf{\partial}^{k+1}(\left\vert A\right\vert )$. Since
$a\geq1$ then $b<\left\vert A\right\vert $ and thus, by induction and
(\ref{shadow vs macaulay}),
\[
\left\vert \partial A_{2}\cup A_{12}\right\vert \geq\left\vert \partial
A_{2}\right\vert \geq\left\vert \partial F_{k+1}\left(  b\right)  \right\vert
=\mathbf{\partial}^{k+1}(b)\text{ and}%
\]%
\[
\left\vert \partial(A_{11}\cup A_{12})\right\vert \geq\left\vert \partial
F_{k}\left(  a\right)  \right\vert =\mathbf{\partial}^{k}(a).
\]
Therefore, by (\ref{macaulay proof identity}), Theorem \ref{technical}, and
(\ref{shadow vs macaulay}); we have%
\[
\left\vert \partial A\right\vert \geq\mathbf{\partial}^{k+1}%
(b)+\mathbf{\partial}^{k}(a)\geq\mathbf{\partial}^{k+1}(\left\vert
A\right\vert )=\left\vert \partial F_{k+1}\left(  \left\vert A\right\vert
\right)  \right\vert .
\]

\end{proof}

In terms of shadows of sets, and using our previous corollary, Theorem
\ref{technical} can be generalized as follows.

\begin{corollary}
Given sets $A\subseteq M_{k}$ and $B\subseteq M_{k+1}$ with $\left\vert
A\right\vert <\left\vert \partial F_{k+1}\left(  \left\vert A\right\vert
+\left\vert B\right\vert \right)  \right\vert $ we have%
\[
\left\vert \partial A\right\vert +\left\vert \partial B\right\vert
\geq\left\vert \partial F_{k+1}\left(  \left\vert A\right\vert +\left\vert
B\right\vert \right)  \right\vert .
\]

\end{corollary}

\begin{proof}
By the previous corollary and (\ref{shadow vs macaulay}), $\left\vert \partial
A\right\vert +\left\vert \partial B\right\vert \geq$ $\mathbf{\partial}%
^{k}\left(  \left\vert A\right\vert \right)  +$ $\mathbf{\partial}%
^{k+1}\left(  \left\vert B\right\vert \right)  $ and $\left\vert A\right\vert
<$ $\mathbf{\partial}^{k+1}\left(  \left\vert A\right\vert +\left\vert
B\right\vert \right)  $. Thus, by Theorem \ref{technical}, $\mathbf{\partial
}^{k}\left(  \left\vert A\right\vert \right)  +$ $\mathbf{\partial}%
^{k+1}\left(  \left\vert B\right\vert \right)  \geq\left\vert \partial
F_{k+1}\left(  \left\vert A\right\vert +\left\vert B\right\vert \right)
\right\vert $.
\end{proof}

The following inequality, proved by Bj\"{o}rner and Vre\'{c}ica, follows
directly from our Theorem. We recall that their proof makes use of Macaulay's
Theorem. Note that $r=1$, $n_{0}=a$, and $n_{1}=n-a$ give the equivalent
inequality to (\ref{kruskal ineq}) for the function $\partial^{k}$.

\begin{corollary}
\emph{(Lemma 3.2 \cite{BS}, also Lemma 2.1 \cite{N}).} For $k>0$, the function
$\partial^{k}$ satisfies that
\begin{align*}
\partial^{k}\left(  \sum_{i=0}^{r}n_{i}\right)   &  \leq\sum_{i=0}^{r}%
\max\left\{  n_{i+1},\partial^{k-i}\left(  n_{i}\right)  \right\}  ,\\
\partial^{k}\left(  1+\sum_{i=0}^{k}n_{i}\right)   &  \leq1+\sum_{i=0}%
^{k-1}\max\left\{  n_{i+1},\partial^{k-i}\left(  n_{i}\right)  \right\}
\text{.}%
\end{align*}
for all nonnegative integers $n_{i}$ and $r<k$. \label{Bjorner}
\end{corollary}

\begin{proof}
By induction on $k$. If $k=1$ the inequalities are trivially true.

Let $r<k+1$, $a=\sum_{i=1}^{r}n_{i}$, and $n=\sum_{i=0}^{r}n_{i}$. If
$a\geq\partial^{k+1}\left(  n\right)  ,$ then%
\begin{align*}
\partial^{k+1}\left(  \sum_{i=0}^{r}n_{i}\right)   &  =\partial^{k+1}\left(
n\right)  \leq a=\sum_{i=1}^{r}n_{i}\leq\sum_{i=0}^{r-1}\max\left\{
n_{i+1},\partial^{k+1-i}\left(  n_{i}\right)  \right\} \\
&  \leq\sum_{i=0}^{r}\max\left\{  n_{i+1},\partial^{k+1-i}\left(
n_{i}\right)  \right\}  \text{.}%
\end{align*}
If, on the other hand, $a<\partial^{k+1}\left(  n\right)  $ then by Theorem
\ref{technical},
\[
\partial^{k+1}\left(  \sum_{i=0}^{r}n_{i}\right)  =\partial^{k+1}\left(
n\right)  \leq\partial^{k+1}\left(  n-a\right)  +\partial^{k}\left(  a\right)
=\partial^{k+1}\left(  n_{0}\right)  +\partial^{k}\left(  \sum_{i=0}%
^{r-1}n_{i+1}\right)  ;
\]
then by induction,%
\begin{align*}
\partial^{k+1}\left(  \sum_{i=0}^{r}n_{i}\right)   &  \leq\partial
^{k+1}\left(  n_{0}\right)  +\sum_{i=0}^{r-1}\max\left\{  n_{i+2}%
,\partial^{k-i}\left(  n_{i+1}\right)  \right\} \\
&  \leq\sum_{i=0}^{r}\max\left\{  n_{i+1},\partial^{k+1-i}\left(
n_{i}\right)  \right\}  \text{.}%
\end{align*}
This proves the first inequality. The second inequality is proved exactly the
same way by letting $a=1+\sum_{i=1}^{k+1}n_{i}$ and $n=1+\sum_{i=0}^{k+1}%
n_{i}$.
\end{proof}

\section{Proof of the theorem}

We first present a simple observation. If $n>k\geq0$ then by Pascal's identity%
\begin{equation}
\binom{n}{k}=\binom{n-1}{k}+\binom{n-2}{k-1}+\cdots+\binom{n-k}{1}%
+\binom{n-k-1}{0}. \label{Pascal identity}%
\end{equation}

Let $a=\sum_{i=1}^{k}\binom{a_{i}}{i}$ be the $k$-binomial representation of
$a$. We say that $a$ is $k$\emph{-long} if $a_{1}\geq1$, and $\emph{k}%
$\emph{-short} if $a_{1}=0$.

\begin{lemma}
\label{short}Let $a\geq0$ be an integer. If $a$ is $k$-short, then
$\mathbf{\partial}^{k}(a+1)=$ $\mathbf{\partial}^{k}(a)+1$, otherwise
$\mathbf{\partial}^{k}(a+1)=$ $\mathbf{\partial}^{k}(a)$.
\end{lemma}

\begin{proof}
The result is clear for $a=0$. If $a\geq1$ is $k$-short, then $a=\sum
_{i=v}^{k}\binom{a_{i}}{i}$ for some $v\geq2$ and $a_{v}\geq v$. Thus
$a+1=\sum_{i=v}^{k}\binom{a_{i}}{i}+\binom{v-1}{v-1}$ is the $k$-binomial
representation of $a+1$ where all the zero terms have been omitted. Then
$\mathbf{\partial}^{k}(a+1)=$ $\mathbf{\partial}^{k}(a)+\binom{v-2}{v-2}=$
$\mathbf{\partial}^{k}(a)+1$.

Now suppose $a$ is $k$-long. There is $v\geq2$ such that $a_{j}=a_{1}+j-1$ for
$j<v$, and either $v=k+1$ or $v\leq k$ and $a_{v}>a_{1}+v-1$. Then%
\[
a+1=\binom{a_{k}}{k}+\cdots+\binom{a_{v}}{v}+\binom{a_{1}+v-2}{v-1}%
+\cdots+\binom{a_{1}+1}{2}+\binom{a_{1}}{1}+\binom{a_{1}-1}{0}%
\]
and by (\ref{Pascal identity}) the $k$-binomial representation of $a+1$ is%
\[
a+1=\binom{a_{k}}{k}+\cdots+\binom{a_{v}}{v}+\binom{a_{1}+v-1}{v-1}\text{.}%
\]
Then, again by (\ref{Pascal identity}),%
\[
\mathbf{\partial}^{k}(a+1)-\mathbf{\partial}^{k}(a)=\binom{a_{1}+v-2}%
{v-2}-\left(  \binom{a_{1}+v-3}{v-2}+\cdots+\binom{a_{1}}{1}+\binom{a_{1}%
-1}{0}\right)  =0\text{.}%
\]

\end{proof}

To prove the Theorem, we need to consider the \emph{extended }$k$%
\emph{-binomial representation} of a positive integer $a$, by requiring an
$a_{0}$ coefficient. That is, we write
\[
a=\binom{a_{a}^{\prime}}{k}+\binom{a_{a-1}^{\prime}}{k-1}+...+\binom
{a_{2}^{\prime}}{2}+\binom{a_{1}^{\prime}}{1}+\binom{a_{0}^{\prime}}%
{0}\text{,}%
\]
with $0\leq a_{0}^{\prime}=a_{1}^{\prime}-1<a_{1}^{\prime}<\cdots
<a_{k}^{\prime}$. The condition $a_{0}^{\prime}=a_{1}^{\prime}-1$ is necessary
to make this representation unique when it exists. Clearly $a=0$ does not have
an extended representation. In general the following is true.

\begin{lemma}
\label{extended}Let $a=\sum_{i=v}^{k}\binom{a_{i}}{i}\geq1$ be the
$k$-binomial representation of $a$, where the terms equal to zero have been
omitted. The extended $k$-binomial representation of $a$ exists (and it is
unique), if and only if $a_{v}\geq v+1$.
\end{lemma}

\begin{proof}
If $a_{v}\geq v+1$, then, by (\ref{Pascal identity}),
\[
\binom{a_{v}}{v}=\binom{a_{v}-1}{v}+\binom{a_{v}-2}{v-1}+\cdots+\binom
{a_{v}-v-1}{0}.
\]
Thus
\[
a=\sum_{i=v+1}^{k}\binom{a_{i}}{i}+\sum_{i=0}^{v}\binom{a_{v}-v-1+i}{i}%
\]
is an extended $k$-representation of $a$. Reciprocally, if $a=\sum_{i=0}%
^{k}\binom{a_{i}^{\prime}}{i}$ is an extended $k$-representation, then
$\binom{a_{0}^{\prime}}{0}=\binom{a_{1}^{\prime}-1}{0},$ and there is $v\geq1$
such that $a_{j}^{\prime}=a_{1}^{\prime}+j-1$ for $0\leq j\leq v$ with either
$v=k$ or $a_{v+1}^{\prime}>a_{1}^{\prime}+v$. Then, by (\ref{Pascal identity}%
),
\[
a=\sum_{i=v+1}^{k}\binom{a_{j}^{\prime}}{j}+\sum_{j=0}^{v}\binom{a_{1}%
^{\prime}+j-1}{j}=\sum_{i=v+1}^{k}\binom{a_{j}^{\prime}}{j}+\binom
{a_{1}^{\prime}+v}{v}\text{,}%
\]
is the $k$-representation of $a$. Thus $a_{v}=a_{1}^{\prime}+v\geq v+1$.
\end{proof}

We can define $\mathbf{\partial}_{e}^{k}(a)=\sum_{i=1}^{k}\binom{a_{i}%
^{\prime}-1}{i-1}$ for the extended $k$-representation of $a$ (if it exists).
It turns out that both definitions agree, i.e., $\mathbf{\partial}^{k}(a)=$
$\mathbf{\partial}_{e}^{k}(a)$. Indeed, if $a=\sum_{i=v}^{k}\binom{a_{i}}{i}$
with $a_{v}\geq v+1$, then by (\ref{Pascal identity}) and the last proof,
\[
\mathbf{\partial}^{k}(a)-\mathbf{\partial}_{e}^{k}(a)=\tbinom{a_{v}-1}%
{v-1}-\sum_{i=0}^{v}\binom{a_{v}-v-2+i}{i-1}=0.
\]

Let $n=\sum_{i=1}^{k+1}\binom{n_{i}}{i},$ $a=\sum_{i=1}^{k}\binom{a_{i}}{i},$
and $n-a=b=\sum_{i=1}^{k+1}\binom{b_{i}}{i}$, be binomial representations.

\begin{lemma}
\label{lemma comparisons2}If $0\leq a<$ $\mathbf{\partial}^{k+1}\left(
n\right)  $, then $a_{k}<n_{k+1}\leq b_{k+1}+1.$
\end{lemma}

\begin{proof}
We prove the contrapositives. If $a_{k}\geq n_{k+1}$, then
\[
a\geq\binom{a_{k}}{k}\geq\binom{n_{k+1}}{k}=\mathbf{\partial}^{k+1}\left(
\binom{n_{k+1}+1}{k+1}\right)  \geq\mathbf{\partial}^{k+1}\left(  n\right)  ,
\]
since $\binom{n_{k+1}+1}{k+1}\geq n$ and $\mathbf{\partial}^{k+1}$ is a
non-decreasing function by Lemma \ref{short}. Now, if $b_{k+1}+1\leq
n_{k+1}-1$, then $b<\binom{b_{k+1}+1}{k+1}\leq\binom{n_{k+1}-1}{k+1}$. Thus%
\[
a=n-b>n-\binom{n_{k+1}-1}{k+1}=\binom{n_{k+1}-1}{k}+\binom{n_{k}}{k}%
+\binom{n_{k-1}}{k-1}+...\binom{n_{1}}{1},
\]
but%
\[
\mathbf{\partial}^{k+1}\left(  n\right)  =\binom{n_{k+1}-1}{k}+\binom{n_{k}%
-1}{k-1}+\binom{n_{k-1}-1}{k-2}+...\binom{n_{1}-1}{0},
\]
and clearly $\binom{n_{i}}{i}\geq\binom{n_{i}-1}{i-1}.$ Thus $a\geq$
$\mathbf{\partial}^{k+1}\left(  n\right)  .$
\end{proof}

\begin{proof}
[Proof of Theorem 1]Recall that $b=n-a$. Clearly, (\ref{technical inequality})
holds if $a=0$, and the case $a=1$ is a consequence of Lemma \ref{short}. We
consider two cases.

\begin{case}
$a_{k}<b_{k+1}$.
\end{case}

Let $a=\sum_{i=v}^{k}\binom{a_{i}}{i}\geq2$ be the $k$-binomial representation
of $a$ without the zero terms. Assume that the pair $(a,b)$ minimizes
$\mathbf{\partial}^{k}(a)+$ $\mathbf{\partial}^{k+1}(b)$ with $a$ as small as possible.

\begin{enumerate}
\item[(i)] Suppose first that $a_{v}\geq v+1$. Then, by Lemma \ref{extended},
$a$ has an extended representation, say $a=\sum_{i=0}^{k}\binom{a_{i}^{\prime
}}{i}$. Let%
\[
\alpha=\sum_{i=1}^{k}\binom{\min(a_{i}^{\prime},b_{i})}{i}\text{ and }%
\beta=\binom{b_{k+1}}{k+1}+\sum_{i=1}^{k}\binom{\max(a_{i}^{\prime},b_{i})}%
{i}+\binom{a_{0}^{\prime}}{0}.
\]
Note that $a+b=\alpha+\beta$ and $\alpha<a$. Also
\[
0\leq\min(a_{1}^{\prime},b_{1})<\min(a_{2}^{\prime},b_{2})<\cdots<\min
(a_{k}^{\prime},b_{k})\text{ and}%
\]%
\[
0\leq a_{0}^{\prime}<\max(a_{1}^{\prime},b_{1})<\cdots<\max(a_{k}^{\prime
},b_{k})<b_{k+1}%
\]
(since $a_{k}^{\prime}\leq a_{k}<b_{k+1}$ by assumption). Therefore the
definitions we gave for $\alpha$ and $\beta$ are $k$-binomial representations
(extended for $\beta$). This means that
\[
\mathbf{\partial}^{k}\left(  \alpha\right)  +\mathbf{\partial}^{k+1}\left(
\beta\right)  =\mathbf{\partial}^{k}\left(  \alpha\right)  +\mathbf{\partial
}_{e}^{k+1}\left(  \beta\right)  =\mathbf{\partial}^{k}\left(  a\right)
+\mathbf{\partial}^{k+1}\left(  b\right)  ,
\]
a contradiction to the minimality of $a$.

\item[(ii)] Assume now that $a_{v}=v$. This means that $a-1=a-\binom{a_{v}}%
{v}=\sum_{i=v+1}^{k}\binom{a_{i}}{i}\geq1$ is the $k$-representation of $a-1$,
and thus $a-1$ is short. Then by Lemma \ref{short},
\[
\mathbf{\partial}^{k}(a-1)+\mathbf{\partial}^{k+1}(b+1)=\mathbf{\partial}%
^{k}(a)-1+\mathbf{\partial}^{k+1}(b+1)\leq\mathbf{\partial}^{k}%
(a)+\mathbf{\partial}^{k+1}(b),
\]
again a contradiction to the minimality of $a$.
\end{enumerate}

Case 1 is settled.

\begin{case}
$b_{k+1}\leq a_{k}$.
\end{case}

Since $a<$ $\mathbf{\partial}^{k+1}\left(  n\right)  $ then, by Lemma
\ref{lemma comparisons2}, $a_{k}<n_{k+1}\leq b_{k+1}+1$. That is,
$a_{k}=b_{k+1}=n_{k+1}-1$. We proceed by induction on $k$. If $k=1,$ then
$a_{1}=b_{2}=n_{2}-1$. Thus
\[
\binom{n_{2}}{2}+\binom{n_{1}}{1}=n=a+b=\binom{n_{2}-1}{1}+\binom{n_{2}-1}%
{2}+\binom{b_{1}}{1},
\]
i.e., $b_{1}=n_{1}$. Hence,
\[
\mathbf{\partial}^{1}\left(  a\right)  +\mathbf{\partial}^{2}\left(  b\right)
=\binom{n_{2}-2}{0}+\binom{n_{2}-2}{1}+\binom{n_{1}-1}{0}=\mathbf{\partial
}^{2}\left(  n\right)  .
\]

Assume $k\geq2$ and that the result holds for $k-1$. Let $n^{\prime}%
=n-\binom{n_{k+1}}{k+1},$ $b^{\prime}=b-\binom{n_{k+1}-1}{k+1},$ and
$a^{\prime}=a-\binom{n_{k+1}-1}{k}$. Clearly, $a^{\prime}+b^{\prime}%
=n^{\prime}$, and $a^{\prime}<$ $\mathbf{\partial}^{k}\left(  n^{\prime
}\right)  $ since $a<$ $\mathbf{\partial}^{k+1}\left(  n\right)
=\binom{n_{k+1}-1}{k}+$ $\mathbf{\partial}^{k}\left(  n^{\prime}\right)  $. By
induction on $k$ the result holds for $a^{\prime},b^{\prime},n^{\prime}$, and
thus%
\begin{align*}
\mathbf{\partial}^{k+1}\left(  b\right)  +\mathbf{\partial}^{k}\left(
a\right)  -\mathbf{\partial}^{k+1}\left(  n\right)   &  =\binom{n_{k+1}-2}%
{k}+\mathbf{\partial}^{k}\left(  b^{\prime}\right)  +\binom{n_{k+1}-2}%
{k-1}+\mathbf{\partial}_{k-1}\left(  a^{\prime}\right) \\
&  -\binom{n_{k+1}-1}{k}+\mathbf{\partial}^{k}\left(  n^{\prime}\right) \\
&  =\mathbf{\partial}^{k}\left(  b^{\prime}\right)  +\mathbf{\partial}%
_{k-1}\left(  a^{\prime}\right)  -\mathbf{\partial}^{k}\left(  n^{\prime
}\right)  \geq0\text{.}%
\end{align*}
Case 2 is now proved.

It is only left to be shown that if $n=\binom{N}{k+1}$ for some $N\geq k+1,$
then the equality in (\ref{technical inequality}), i.e.
\begin{equation}
\mathbf{\partial}^{k}\left(  a\right)  +\mathbf{\partial}^{k+1}\left(
\binom{N}{k+1}-a\right)  =\mathbf{\partial}^{k+1}\left(  \binom{N}%
{k+1}\right)  , \label{equality}%
\end{equation}
occurs only when $a=0.$ If $N=k+1$ (and thus $a=0)$ or $a=0,$ the equality
trivially holds. Suppose that $N\geq k+2$, $a\geq1$, and (\ref{equality})
holds. Let $b=\binom{N}{k+1}-a$; as before, we consider two cases. First
suppose that $a_{k}<b_{k+1}$. Assume that $a$ and $b$ are the smallest
integers such that (\ref{equality}) is satisfied with $a\geq1$. If $a=1$, then
by (\ref{Pascal identity}),
\[
\binom{N}{k+1}-1=\binom{N-1}{k+1}+\binom{N-2}{k}+\cdots+\binom{N-k-1}{1}%
\]
is $(k+1)$-long. Thus, by Lemma \ref{short}, $\mathbf{\partial}^{k+1}%
(\binom{N}{k+1}-1)=\mathbf{\partial}^{k+1}\left(  \binom{N}{k+1}\right)
<1+\mathbf{\partial}^{k+1}(\binom{N}{k+1}-1)$, which is a contradiction. If
$a\geq2$, then we proceed as in Case 1 to get a contradiction. Now assume that
$b_{k+1}\leq a_{k}$. In this case, $a_{k}<N\leq b_{k+1}+1$ and following the
procedure of Case 2, we have that $a_{k}=b_{k+1}=N-1,$ a contradiction since
$\binom{N-1}{k}=\binom{a_{k}}{k}\leq a<\mathbf{\partial}^{k+1}\left(
\binom{N}{k+1}\right)  =\binom{N-1}{k}$.
\end{proof}

\textbf{Acknowledgements}. We are grateful to Hans Fetter and A. Paulina
Figueroa for their valuable help during the preparation of this paper as well
as to Eduardo Rivera-Campo for useful comments and suggestions.

\end{document}